\documentclass[12pt,reqno,twoside]{amsart}
\usepackage{amsmath}
\usepackage{amssymb,amsfonts,amsthm,indentfirst}
\usepackage{epsfig}
\usepackage{graphicx}
\usepackage{hyperref}

\newtheorem{theorem}{Theorem}[section]
\newtheorem{proposition}{Proposition}[section]

\newtheorem{lemma}{Lemma}[section]

\newtheorem{Remark}{Remark}[section]

\newcommand{\be}{\begin{equation}}
\newcommand{\ee}{\end{equation}}
\newcommand{\bea}{\begin{eqnarray}}
\newcommand{\eea}{\end{eqnarray}}
\newcommand{\eeas}{\end{eqnarray*}}
\newcommand{\beas}{\begin{eqnarray*}}

\numberwithin{equation}{section}

\begin{document}

\begin{center}
{\large \textbf{\ On Interpolating Sesqui-Harmonic Legendre Curves in
Sasakian Space Forms}}

\bigskip \bigskip

\textbf{Fatma KARACA}, \textbf{Cihan \"{O}ZG\"{U}R and Uday Chand DE}
\bigskip
\end{center}

\textbf{Abstract.} We consider interpolating sesqui-harmonic Legendre curves
in Sasakian space forms. We find the necessary and sufficient conditions for
Legendre curves in Sasakian space forms to be interpolating sesqui-harmonic.
Finally, we obtain an example for an interpolating sesqui-harmonic Legendre
curve in a Sasakian space form.



\medskip

\section{\textbf{Introduction }\label{sect-introduction}}

A map $\varphi $ $:\left( M,g\right) \rightarrow \left( N,h\right) $ between
Riemannian manifolds is called a \textit{harmonic map }and \ a \textit{%
biharmonic map, }respectively if it is a critical point of the $E(\varphi )$%
\textit{\ }and $E_{2}(\varphi )$
\begin{equation*}
E(\varphi )=\int_{\Omega }\left\Vert d\varphi \right\Vert ^{2}d\nu _{g},
\end{equation*}%
\begin{equation*}
E_{2}(\varphi )=\int_{\Omega }\left\Vert \tau (\varphi )\right\Vert ^{2}d\nu
_{g},
\end{equation*}%
where $\Omega $ is a compact domain of $M$. The harmonic map equation is%
\begin{equation}
\tau (\varphi )=tr(\nabla d\varphi )=0,  \label{harmonic}
\end{equation}%
and it is called the \textit{tension field }of $\varphi $ \cite{ES}$.$ The
Euler-Lagrange equation of $E_{2}(\varphi )$ is
\begin{equation}
\tau _{2}(\varphi )=tr(\nabla ^{\varphi }\nabla ^{\varphi }-\nabla _{\nabla
}^{\varphi })\tau (\varphi )-tr(R^{N}(d\varphi ,\tau (\varphi ))d\varphi )=0,
\label{bitensionfield}
\end{equation}%
and it is called the \textit{bitension field} of $\varphi $ \cite{Ji1}.

In \cite{Branding}, Branding defined and considered interpolating
sesqui-harmonic maps between Riemannian manifolds. The author introduced an
action functional for maps between Riemannian manifolds that interpolated
between the actions for harmonic and biharmonic maps. The map $\varphi $ is
said to be \textit{interpolating sesqui-harmonic }if it is a critical point
of $E_{\delta _{1},\delta _{2}}(\varphi )$ \textit{\ }
\begin{equation}
E_{\delta _{1},\delta _{2}}(\varphi )=\delta _{1}\int_{\Omega }\left\Vert
d\varphi \right\Vert ^{2}d\nu _{g}+\delta _{2}\int_{\Omega }\left\Vert \tau
(\varphi )\right\Vert ^{2}d\nu _{g},  \label{criticalpointsemibiharmonic}
\end{equation}%
where $\Omega $ is a compact domain of $M$ and $\delta _{1},\delta _{2}\in
\mathbb{R}
$ \cite{Branding}. The interpolating sesqui-harmonic map equation is%
\begin{equation}
\tau _{\delta _{1},\delta _{2}}(\varphi )=\delta _{2}\tau _{2}(\varphi
)-\delta _{1}\tau (\varphi )=0  \label{semibiharmonic}
\end{equation}%
for $\delta _{1},\delta _{2}\in
\mathbb{R}
$ \cite{Branding}. An interpolating sesqui-harmonic map is biminimal if
variations of (\ref{criticalpointsemibiharmonic}) that are normal to the
image $\varphi (M)\subset N$ and $\delta _{2}=1$, $\delta _{1}>0$ \cite{LM}.
For some recent study of biminimal immersions see \cite{GO-17}, \cite{LM},
\cite{Luo} and \cite{Maeta}.

Interpolating sesqui-harmonic curves in a $3$-dimensional sphere were
studied in \cite{Branding}. In \cite{Fetcu} and \cite{FO}, Fetcu and Oniciuc
considered biharmonic Legendre curves in Sasakian space forms. In \cite%
{Cho-09}, Cho, Inoguchi and Lee studied affine biharmonic curves in $3$%
-dimensional pseudo-Hermitian geometry. In \cite{In-13}, Inoguchi and Lee
studied affine biharmonic curves in $3$-dimensional homogeneous geometries.
\ In \cite{OG}, the second author and G\"{u}ven\c{c} studied biharmonic
Legendre curves in generalized Sasakian space forms. In \cite{GO}, G\"{u}ven%
\c{c} and the second author studied $f$-biharmonic Legendre curves in
Sasakian space forms. Motivated by the above studies, in the present paper,
we consider interpolating sesqui-harmonic Legendre curves in Sasakian space
forms. We obtain the necessary and sufficient conditions for Legendre curves
in Sasakian space forms to be interpolating sesqui-harmonic. We also give an
example for an interpolating sesqui-harmonic Legendre curve in a Sasakian
space form.

\section{\textbf{Preliminaries}}

Let $M=(M^{2n+1},\phi ,\xi ,\eta ,g)$ be an almost contact metric manifold
with an almost contact metric structure $(\phi ,\xi ,\eta ,g)$. A contact
metric manifold $(M^{2n+1},\phi ,\xi ,\eta ,g)$ is called a \textit{Sasakian
manifold} if it is normal, that is,%
\begin{equation*}
N_{\phi }=-2d\eta \otimes \xi
\end{equation*}%
where $N_{\phi }$ is the Nijenhuis tensor field of $\phi $ \cite{Blair}. It
is well-known that an almost contact metric manifold is Sasakian if and only
if%
\begin{equation*}
\left( \nabla _{X}\phi \right) Y=g(X,Y)\xi -\eta \left( Y\right) X
\end{equation*}%
and%
\begin{equation*}
\nabla _{X}\xi =-\phi X
\end{equation*}%
\cite{Blair2}. The sectional curvature of a $\phi $-section is called a $%
\phi $-\textit{sectional curvature}. When the $\phi $-sectional curvature is
a constant, then the Sasakian manifold is called a\textit{\ Sasakian space
form }and it is denoted by\textit{\ }$M(c)$\textit{\ }\cite{Blair2}. The
curvature tensor $R$ of a Sasakian space form $M(c)$ is given by
\begin{equation*}
R\left( X,Y\right) Z=\frac{c+3}{4}\left\{ g(Y,Z)X-g(X,Z)Y\right\}
\end{equation*}%
\begin{equation*}
+\frac{c-1}{4}\left\{ g(X,\phi Z)\phi Y-g(Y,\phi Z)\phi X+2g(X,\phi Y)\phi
Z\right.
\end{equation*}%
\begin{equation}
\left. +\eta \left( X\right) \eta \left( Z\right) Y-\eta \left( Y\right)
\eta \left( Z\right) X+g(X,Z)\eta \left( Y\right) \xi -g(Y,Z)\eta \left(
X\right) \xi \right\}  \label{curvature}
\end{equation}%
for all $X,Y,Z$ $\in TM$ \cite{Blair2}.

A submanifold of a Sasakian manifold $M$ is called an \textit{integral
submanifold} if $\eta (X)=0$, for every tangent vector $X$. An integral
curve of a Sasakian manifold $M$ is called a \textit{Legendre curve} \cite%
{Blair2}.

\section{\textbf{Interpolating sesqui-harmonic Legendre curves in Sasakian
space forms \label{sesqui harmonic legendre curves}}}

Let $\gamma :I\subset
\mathbb{R}
\longrightarrow (M^{n},g)$ be a curve parametrized by arc length in a
Riemannian manifold $(M^{n},g)$ . Then $\gamma $ is called a Frenet curve of
osculating order $r$, $1\leq r\leq n$, if there exists orthonormal vector
fields $\left\{ E_{i}\right\} _{i=1,2,...n}$ along $\gamma $ such that%
\begin{equation*}
E_{1}=T=\gamma ^{\prime },
\end{equation*}%
\begin{equation*}
\nabla _{T}E_{1}=k_{1}E_{2},
\end{equation*}%
\begin{equation}
\nabla _{T}E_{i}=-k_{i-1}E_{i-1}+k_{i}E_{i+1},\text{ \ }2\leq i\leq n-1,
\label{frenetframe}
\end{equation}

\begin{equation*}
\nabla _{T}E_{n}=-k_{n-1}E_{n-1},
\end{equation*}%
where the function $\left\{ k_{1}=k,k_{2}=\tau ,k_{3},...,k_{n-1}\right\} $
are called the curvatures of $\gamma $ \cite{La}.

Firstly, we have the following theorem for an interpolating sesqui-harmonic
Legendre curve in a Sasakian space form:

\begin{theorem}
\label{Theo 3.1}Let $M(c)=(M^{2n+1},\phi ,\xi ,\eta ,g)$ be a Sasakian space
form with constant $\phi $-sectional curvature $c$ and $\gamma :I\subset
\mathbb{R}
\longrightarrow M(c)$ be a Legendre curve of osculating order $r$ and $%
m=min\{r,4\}$. Then $\gamma $ is interpolating sesqui-harmonic if and only
if there exists real numbers $\delta _{1},\delta _{2}$ such that

$\left( 1\right) $ $c=1$ or $\phi T\perp E_{2}$ or $\phi T\in \left\{
E_{2},...,E_{m}\right\} ;$ and

$\left( 2\right) $ the first $m$ of the following equations are satisfied:
\begin{equation}
-3\delta _{2}k_{1}k_{1}^{\prime }=0,  \label{semibiharmonic1}
\end{equation}%
\begin{equation*}
\delta _{2}\left[ k_{1}^{\prime \prime }-k_{1}^{3}-k_{1}k_{2}^{2}-\left(
\frac{c+3}{4}\right) k_{1}\right.
\end{equation*}%
\begin{equation}
\left. +3\left( \frac{c-1}{4}\right) k_{1}\left[ g(\phi T,E_{2})\right]
^{2}-\left( \frac{c-1}{4}\right) k_{1}\left[ \eta \left( E_{2}\right) \right]
^{2}\right] -\delta _{1}k_{1}=0,  \label{semibiharmonic2}
\end{equation}%
\begin{equation*}
\delta _{2}\left[ 2k_{1}^{\prime }k_{2}+k_{1}k_{2}^{\prime }+3\left( \frac{%
c-1}{4}\right) k_{1}g(\phi T,E_{2})g(\phi T,E_{3})\right.
\end{equation*}%
\begin{equation}
\left. -\left( \frac{c-1}{4}\right) k_{1}\eta \left( E_{2}\right) \eta
\left( E_{3}\right) \right] =0,  \label{semibiharmonic3}
\end{equation}%
\begin{equation*}
\delta _{2}\left[ k_{1}k_{2}k_{3}+3\left( \frac{c-1}{4}\right) k_{1}g(\phi
T,E_{2})g(\phi T,E_{4})\right.
\end{equation*}%
\begin{equation}
\left. -\left( \frac{c-1}{4}\right) k_{1}\eta \left( E_{2}\right) \eta
\left( E_{4}\right) \right] =0.  \label{semibiharmonic4}
\end{equation}
\end{theorem}

\begin{proof}
\bigskip Let $\gamma :I\longrightarrow M$ be a Legendre curve of osculating
order $r$ in $M(c).$ $\ $By the use of (\ref{harmonic}) and (\ref%
{frenetframe}), we have
\begin{equation}
\tau (\gamma )=\nabla _{T}T=k_{1}E_{2}.  \label{1}
\end{equation}%
From (\ref{frenetframe}), we get%
\begin{equation*}
\nabla _{T}\nabla _{T}T=-k_{1}^{2}E_{1}+k_{1}^{\prime }E_{2}+k_{1}k_{2}E_{3},
\end{equation*}%
\begin{equation*}
\nabla _{T}\nabla _{T}\nabla _{T}T=-3k_{1}k_{1}^{^{\prime }}E_{1}+\left(
k_{1}^{\prime \prime }-k_{1}^{3}-k_{1}k_{2}^{2}\right) E_{2}
\end{equation*}%
\begin{equation}
+\left( 2k_{1}^{\prime }k_{2}+k_{1}k_{2}^{\prime }\right) E_{3}+\left(
k_{1}k_{2}k_{3}\right) E_{4},  \label{2}
\end{equation}%
\begin{equation*}
R(T,\nabla _{T}T)T=-\left( \frac{c+3}{4}\right) k_{1}E_{2}
\end{equation*}%
\begin{equation}
-3\left( \frac{c-1}{4}\right) k_{1}g(\phi T,E_{2})\phi T+\left( \frac{c-1}{4}%
\right) k_{1}\eta \left( E_{2}\right) \xi .  \label{3}
\end{equation}%
Using the equations (\ref{1}), (\ref{2}) and (\ref{3}) into the equation
(4.1) in \cite{Branding}, we find%
\begin{equation*}
\tau _{\delta _{1},\delta _{2}}(\gamma )=\left( -3\delta
_{2}k_{1}k_{1}^{\prime }\right) E_{1}+\left[ \delta _{2}\left( k_{1}^{\prime
\prime }-k_{1}^{3}-k_{1}k_{2}^{2}+\left( \frac{c+3}{4}\right) k_{1}\right)
-\delta _{1}k_{1}\right] E_{2}
\end{equation*}%
\begin{equation*}
+\delta _{2}\left( 2k_{1}^{\prime }k_{2}+k_{1}k_{2}^{\prime }\right)
E_{3}+\delta _{2}\left( k_{1}k_{2}k_{3}\right) E_{4}
\end{equation*}%
\begin{equation}
+3\left( \frac{c-1}{4}\right) \delta _{2}k_{1}g(\phi T,E_{2})\phi T-\left(
\frac{c-1}{4}\right) \delta _{2}k_{1}\eta \left( E_{2}\right) \xi .
\label{semibiharmoniceq1}
\end{equation}%
Taking the scalar product of equation (\ref{semibiharmoniceq1}) with $%
E_{2},E_{3}$ and $E_{4}$ respectively, then we obtain the desired result.
\end{proof}

Now we shall discuss some special cases of Theorem \ref{Theo 3.1}:

\textbf{Case I. }$c=1.$

From Theorem \ref{Theo 3.1}, \ we have:

\begin{proposition}
\label{theo3.2}Let $M(1)=(M^{2n+1},\phi ,\xi ,\eta ,g)$ be a Sasakian space
form with $c=1$ and $\gamma :I\subset
\mathbb{R}
\longrightarrow M(1)$ be a Legendre curve of osculating order $r$ \textit{%
such that }$\frac{\delta _{1}}{\delta _{2}}\neq 0$. Then\textit{\ }$\gamma $%
\textit{\ is }interpolating sesqui-harmonic \textit{if and only if }
\begin{equation*}
k_{1}=\text{constant}>0,\text{ }k_{2}=\text{constant,}
\end{equation*}%
\begin{equation*}
k_{1}^{2}+k_{2}^{2}=1-\frac{\delta _{1}}{\delta _{2}},
\end{equation*}%
\begin{equation*}
k_{2}k_{3}=0
\end{equation*}%
\textit{where }$1-\frac{\delta _{1}}{\delta _{2}}>0,$ $\delta _{1},$ $\delta
_{2}$\textit{\ is a constant.}
\end{proposition}

\begin{proof}
Assume that $\gamma $ is an interpolating sesqui-harmonic Legendre curve of
osculating order $r$ in $M(1)$ such that $\frac{\delta _{1}}{\delta _{2}}%
\neq 0$ and $c=1.$ From Theorem \ref{Theo 3.1}, we obtain the result.
\end{proof}

Using Proposition \ref{theo3.2}, we have:

\begin{theorem}
\label{theo2}Let $M(1)=(M^{2n+1},\phi ,\xi ,\eta ,g)$ be a Sasakian space
form with $c=1$ and $\gamma :I\subset
\mathbb{R}
\longrightarrow M(1)$ be a non geodesic Legendre curve of osculating order $%
r $. Then

$(1)$ It is a Legendre geodesic or

$(2)$\textit{\ }$\gamma $\textit{\ is interpolating sesqui-harmonic with }$%
\frac{\delta _{1}}{\delta _{2}}\neq 0$ \textit{if and only if it is a
Legendre circle with }$k_{1}=\sqrt{1-\frac{\delta _{1}}{\delta _{2}}}$
\textit{where }$1-\frac{\delta _{1}}{\delta _{2}}>0$\textit{\ is a constant
or}

$(3)$ $\gamma $\textit{\ is interpolating sesqui-harmonic with }$\frac{%
\delta _{1}}{\delta _{2}}\neq 0$ \textit{if and only if it} \textit{is a
Legendre helix with }$k_{1}^{2}+k_{2}^{2}=1-\frac{\delta _{1}}{\delta _{2}}$
\textit{where }$1-\frac{\delta _{1}}{\delta _{2}}>0,$ $\delta _{1},$ $\delta
_{2}$\textit{\ is a constant.}

\textit{In both cases, if }$1-\frac{\delta _{1}}{\delta _{2}}<0$\textit{,
then such an interpolating sesqui-harmonic Legendre curve does not exist.}
\end{theorem}

\begin{proof}
Let $\gamma :I\longrightarrow M(1)$ be an interpolating sesqui-harmonic
curve with\textit{\ }$\frac{\delta _{1}}{\delta _{2}}\neq 0$. From Theorem %
\ref{theo3.2}, if we consider the osculating order $r=2$, then $\gamma $ is
a Legendre circle with $k_{1}=\sqrt{1-\frac{\delta _{1}}{\delta _{2}}}$
where $1-\frac{\delta _{1}}{\delta _{2}}>0$ is a constant.\textit{\ }%
Similarly, if we consider the osculating order $r=3$, then we obtain that $%
k_{2}$ is a non-zero constant. Thus, $\gamma $ is a Legendre helix with $%
k_{1}^{2}+k_{2}^{2}=1-\frac{\delta _{1}}{\delta _{2}}$ where $1-\frac{\delta
_{1}}{\delta _{2}}>0$ is a constant. On the other hand, assume that $\gamma $
is a Legendre circle with $k_{1}=\sqrt{1-\frac{\delta _{1}}{\delta _{2}}}$
or a Legendre helix with $k_{1}^{2}+k_{2}^{2}=1-\frac{\delta _{1}}{\delta
_{2}}$ where $1-\frac{\delta _{1}}{\delta _{2}}>0$ is a constant. Obviously,
$\gamma $ satisfies Theorem \ref{Theo 3.1}, respectively. It is trivial that
$1-\frac{\delta _{1}}{\delta _{2}}<0$ cannot be possible. If $1-\frac{\delta
_{1}}{\delta _{2}}=0,$ we obtain a geodesic. This proves the theorem.
\end{proof}

\textbf{Case II.} $c\neq 1$ and $\phi T\perp E_{2}.$

From Theorem \ref{Theo 3.1}, \ we can state:

\begin{proposition}
\label{theo3.4}Let $M(c)=(M^{2n+1},\phi ,\xi ,\eta ,g)$ be a Sasakian space
form with $c\neq 1,$ $\phi T\perp E_{2}$ and $\gamma :I\subset
\mathbb{R}
\longrightarrow M(c)$ be a Legendre curve of osculating order $r$ \textit{%
such that }$\frac{\delta _{1}}{\delta _{2}}\neq 0$. Then $\gamma $\textit{\
is }interpolating sesqui-harmonic \textit{if and only if }
\begin{equation*}
k_{1}=\text{constant}>0,\text{ }k_{2}=\text{constant,}
\end{equation*}%
\begin{equation*}
k_{1}^{2}+k_{2}^{2}=\frac{c+3}{4}-\frac{\delta _{1}}{\delta _{2}},
\end{equation*}%
\begin{equation*}
k_{2}k_{3}=0
\end{equation*}%
\textit{where} $\delta _{1},$ $\delta _{2}$\textit{\ is a constant.}
\end{proposition}

\begin{proof}
Let $\gamma $ be an interpolating sesqui-harmonic Legendre curve of
osculating order $r$ in $M(c)$ such that $c\neq 1,$ $\phi T\perp E_{2}$ and $%
\frac{\delta _{1}}{\delta _{2}}\neq 0.$ From Theorem \ref{Theo 3.1}, we get
the result.
\end{proof}

From \cite{FO}, we have the following lemma:

\begin{lemma}
\label{lemma1}\cite{FO} Let $\gamma $ be a Legendre Frenet curve of
osculating order $3$ in a Sasakian space form $M(c)$ and $\phi T\perp E_{2}$%
. Then $\left\{ T=E_{1},E_{2},E_{3},\phi T,\nabla _{T}\phi T,\xi \right\} $
is linearly independent at any point of $\gamma $ and therefore $n\geq 3$.
\end{lemma}

Hence we can state:

\begin{theorem}
Let $M(c)=(M^{2n+1},\phi ,\xi ,\eta ,g)$ be a Sasakian space form with $%
c\neq 1,$ $\phi T\perp E_{2}$ and $\gamma :I\subset
\mathbb{R}
\longrightarrow M(c)$ a Legendre curve of osculating order $r$.

$(1)$ If $c\leq -3$ and $\frac{\delta _{1}}{\delta _{2}}$ $\geq 0$, then $%
\gamma $ is interpolating sesqui-harmonic if and only if it is a geodesic.

$(2)$ If $c>-3$ and $\frac{\delta _{1}}{\delta _{2}}$ $<0$, then $\gamma $
is interpolating sesqui-harmonic if and only if either

$(a)$ $\gamma $ is of osculating order $r=2$, $n\geq 2$ and $\gamma $\textit{%
\ is a circle with }$k_{1}^{2}=\frac{c+3}{4}-\frac{\delta _{1}}{\delta _{2}}%
, $ in which case $\left\{ T,E_{2},\phi T,\nabla _{T}\phi T,\xi \right\} $
are linearly independent, or\textit{\ }

$(b)$ $\gamma $ is of osculating order $r=3$, $n\geq 3$ and $\gamma $\textit{%
\ is a helix with }$k_{1}^{2}+k_{2}^{2}=\frac{c+3}{4}-\frac{\delta _{1}}{%
\delta _{2}},$ in which case $\left\{ T,E_{2},E_{3},\phi T,\nabla _{T}\phi
T,\xi \right\} $ are linearly independent, where $\delta _{1},$ $\delta _{2}$%
\textit{\ }$\in
\mathbb{R}
$\textit{.}
\end{theorem}

\begin{proof}
$(1)$ From Proposition \ref{theo3.4}, if we take $c\leq -3$ and $\frac{%
\delta _{1}}{\delta _{2}}$ $\geq 0$, it is easy to see that $\gamma $ is
interpolating sesqui-harmonic if and only if it is a geodesic.

$(2)$ Assume that $c>-3$, $\frac{\delta _{1}}{\delta _{2}}$ $<0$ and $\gamma
:I\longrightarrow M(c)$ be an interpolating sesqui-harmonic curve. From
Proposition \ref{theo3.4}, if we take $n\geq 2$ and $\gamma $ is of
osculating order $r=2$, then $\gamma $ is a circle with $k_{1}^{2}=\frac{c+3%
}{4}-\frac{\delta _{1}}{\delta _{2}}.$ Using Lemma \ref{lemma1}, we have
that $\left\{ T,E_{2},\phi T,\nabla _{T}\phi T,\xi \right\} $ are linearly
independent. Similarly, if we take $n\geq 3$ and $\gamma $ is of osculating
order $r=3$, then we obtain that $k_{2}$ is a non-zero constant. Thus, $%
\gamma $ is a helix with $k_{1}^{2}+k_{2}^{2}=\frac{c+3}{4}-\frac{\delta _{1}%
}{\delta _{2}}.$ Using Lemma \ref{lemma1}, we have that $\left\{
T,E_{2},E_{3},\phi T,\nabla _{T}\phi T,\xi \right\} $ are linearly
independent. Conversely, assume that $\gamma $ is a Legendre circle with $%
k_{1}^{2}=\frac{c+3}{4}-\frac{\delta _{1}}{\delta _{2}}$ or a Legendre helix
with $k_{1}^{2}+k_{2}^{2}=\frac{c+3}{4}-\frac{\delta _{1}}{\delta _{2}}$.
Obviously, $\gamma $ satisfies Theorem \ref{Theo 3.1}, respectively. Hence,
we obtain the desired result.
\end{proof}

\textbf{Case III. }$c\neq 1$ and $\phi T\parallel E_{2}.$

From Theorem \ref{Theo 3.1}, \ we have:

\begin{proposition}
\label{theo3.6}Let $M(c)=(M^{2n+1},\phi ,\xi ,\eta ,g)$ be a Sasakian space
form with $c\neq 1$ and $\gamma :I\subset
\mathbb{R}
\longrightarrow M(c)$ be a Legendre curve of osculating order $r$ \textit{%
with }$\phi T\parallel E_{2}$ and\textit{\ }$\frac{\delta _{1}}{\delta _{2}}%
\neq 0$. Then $\gamma $\textit{\ is }interpolating sesqui-harmonic \textit{%
if and only if }
\begin{equation*}
k_{1}=\text{constant}>0,\text{ }k_{2}=\text{constant,}
\end{equation*}%
\begin{equation*}
k_{1}^{2}+k_{2}^{2}=c-\frac{\delta _{1}}{\delta _{2}},
\end{equation*}%
\begin{equation*}
k_{2}k_{3}=0
\end{equation*}%
\textit{where} $\delta _{1},$ $\delta _{2}$\textit{\ is a constant.}
\end{proposition}

\begin{proof}
Assume $\gamma $ is an interpolating sesqui-harmonic Legendre curve in $M(c)$
such that $c\neq 1$, $\phi T\parallel E_{2}$ and $\frac{\delta _{1}}{\delta
_{2}}\neq 0.$ From Theorem \ref{Theo 3.1}, we get the result.
\end{proof}

Hence we can state:

\begin{theorem}
Let $M(c)=(M^{2n+1},\phi ,\xi ,\eta ,g)$ be a Sasakian space form with $%
c\neq 1$ and $\gamma :I\subset
\mathbb{R}
\longrightarrow M(c)$ a Legendre curve of osculating order $r$ such that $%
\phi T\parallel E_{2}$. Then $\left\{ T,\phi T,\xi \right\} $ is the Frenet
frame field of $\gamma .$

$(1)$ If $c<1$ and $\frac{\delta _{1}}{\delta _{2}}$ $\geq 0$, then $\gamma $
is interpolating sesqui-harmonic if and only if it is a geodesic.

$(2)$ If $c>1$ and $\frac{\delta _{1}}{\delta _{2}}$ $<0$, then $\gamma $ is
interpolating sesqui-harmonic if and only if it\textit{\ is a helix with }$%
k_{1}^{2}=c-1-\frac{\delta _{1}}{\delta _{2}},$ $\left( k_{2}=1\right) $
where $\delta _{1},$ $\delta _{2}$\textit{\ }$\in
\mathbb{R}
$\textit{.}
\end{theorem}

\begin{proof}
If we take $\phi T\parallel E_{2},$ we get $g\left( \phi T,E_{2}\right) =\pm
1,$ $g\left( \phi T,E_{3}\right) =g\left( \phi T,E_{4}\right) =0.$

$(1)$ From Proposition \ref{theo3.6} and the above equations and if we take $%
c\leq 1$ and $\frac{\delta _{1}}{\delta _{2}}$ $\geq 0$, it is easy to see
that $\gamma $ is interpolating sesqui-harmonic if and only if it is a
geodesic.

$(2)$ If $c>1$, $\frac{\delta _{1}}{\delta _{2}}$ $<0$ from Proposition \ref%
{theo3.6} and the above equations, we have $k_{1}=$constant and $%
k_{1}^{2}=c-1-\frac{\delta _{1}}{\delta _{2}},$ and $k_{2}=1$. Conversely,
assume that $\gamma $ is a Legendre helix with $k_{1}^{2}=c-1-\frac{\delta
_{1}}{\delta _{2}}$ and $k_{2}=1$. Then $\gamma $ satisfies Theorem \ref%
{Theo 3.1} obviously. This completes the proof of the theorem.
\end{proof}

\textbf{Case IV. }$c\neq 1$ and $g(\phi T,E_{2})\neq 0,1,-1$.

\begin{proposition}
\label{theov}Let $M(c)=(M^{2n+1},\phi ,\xi ,\eta ,g)$ be a Sasakian space
form with $c\neq 1,$ $g(\phi T,E_{2})\neq 0,1,-1$ and $\gamma :I\subset
\mathbb{R}
\longrightarrow M(c)$ a Legendre curve of osculating order $r$ such that $%
4\leq r\leq 2n+1$, $n\geq 2$. Then $\gamma $\textit{\ is interpolating
sesqui-harmonic with }$\frac{\delta _{1}}{\delta _{2}}\neq 0$\textit{\ if
and only if }%
\begin{equation*}
k_{1}=\text{\textit{constant}}>0,
\end{equation*}%
\begin{equation*}
k_{1}^{2}+k_{2}^{2}=\frac{c+3}{4}+\frac{3\left( c-1\right) }{4}f^{2}-\frac{%
\delta _{1}}{\delta _{2}},
\end{equation*}%
\begin{equation*}
k_{2}^{\prime }=-\frac{3\left( c-1\right) }{4}fg\left( E_{3},\phi T\right) ,
\end{equation*}%
\begin{equation*}
k_{2}k_{3}=-\frac{3\left( c-1\right) }{4}fg\left( E_{4},\phi T\right) .
\end{equation*}
\end{proposition}

\begin{proof}
Assume that $\gamma $ is an interpolating sesqui-harmonic Legendre Frenet
curve such that $g(\phi T,E_{2})$ is not a constant equal to $0,1$ or $-1$.
In this case, we get $4\leq r\leq 2n+1$, $n\geq 2$ and $\phi T\in
span\left\{ E_{2},E_{3},E_{4}\right\} $.

Hence, we can take $f(t)=$ $g(\phi T,E_{2}).$ So by a differentiation, we
obtain%
\begin{equation*}
f^{\prime }(t)=g(\nabla _{T}\phi T,E_{2})+g(\phi T,\nabla _{T}E_{2})
\end{equation*}%
\begin{equation*}
=-k_{1}g\left( T,\phi T\right) +k_{2}g\left( E_{3},\phi T\right) +g\left(
E_{2},\xi \right) +k_{1}g\left( E_{2},\phi E_{2}\right) .
\end{equation*}%
Since $\gamma $ is a Legendre curve and $\phi $ is anti-symmetric, we have $%
\eta (E_{2})=0,$ $g\left( T,\phi T\right) =0$ and $g\left( E_{2},\phi
E_{2}\right) =0.$ Thus we obtain%
\begin{equation}
f^{\prime }(s)=k_{2}g\left( E_{3},\phi T\right) .  \label{f}
\end{equation}%
Additionally, we can write%
\begin{equation}
\phi T=g\left( \phi T,E_{2}\right) E_{2}+g\left( \phi T,E_{3}\right)
E_{3}+g\left( \phi E_{4},E_{4}\right) E_{4}.  \label{fiT}
\end{equation}%
From Theorem \ref{Theo 3.1}, the equations (\ref{f}) and (\ref{fiT}), the
curve $\gamma $ is interpolating sesqui-harmonic if and only if
\begin{equation*}
k_{1}=\text{constant,}
\end{equation*}%
\begin{equation*}
k_{1}^{2}+k_{2}^{2}=\frac{c+3}{4}+\frac{3\left( c-1\right) }{4}f^{2}-\frac{%
\delta _{1}}{\delta _{2}},
\end{equation*}%
\begin{equation*}
k_{2}^{\prime }=-\frac{3\left( c-1\right) }{4}fg\left( E_{3},\phi T\right) ,
\end{equation*}%
\begin{equation*}
k_{2}k_{3}=-\frac{3\left( c-1\right) }{4}fg\left( E_{4},\phi T\right) .
\end{equation*}%
If\ $\gamma :I\subset
\mathbb{R}
\longrightarrow M(c)$ satisfies the converse statement, it is obvious that
the first four of the equations in Theorem \ref{Theo 3.1} are satisfied.
Thus $\gamma $ is interpolating sesqui-harmonic. This proves the theorem.
\end{proof}

Using the equation (\ref{f}) and the third equation of Proposition \ref%
{theov}, we obtain%
\begin{equation*}
k_{2}^{\prime }=-\frac{3\left( c-1\right) }{4}fg\left( E_{3},\phi T\right) =-%
\frac{3\left( c-1\right) }{4}f\frac{f^{\prime }}{k_{2}}
\end{equation*}%
\begin{equation*}
k_{2}k_{2}^{\prime }=-\frac{3\left( c-1\right) }{4}ff^{\prime }
\end{equation*}%
\begin{equation}
k_{2}^{2}=-\frac{3\left( c-1\right) }{4}f^{2}+w_{0}  \label{k2}
\end{equation}%
where $w_{0}=$constant. Substituting the equation (\ref{k2}) in the second
equation of Proposition \ref{theov}, we get%
\begin{equation*}
k_{1}^{2}=\frac{c+3}{4}+\frac{3\left( c-1\right) }{2}f^{2}-\frac{\delta _{1}%
}{\delta _{2}}-w_{0}.
\end{equation*}%
Then we have $f=$constant. Thus $k_{2}=$constant $>0$, $g\left( E_{3},\phi
T\right) =0$ and then $\phi T=fE_{2}+g\left( \phi T,E_{4}\right) E_{4}$. We
obtain that there exists a unique constant $\alpha _{0}\in \left( 0,2\pi
\right) \backslash \left\{ \frac{\pi }{2},\pi ,\frac{3\pi }{2}\right\} $
such that $f=\cos \alpha _{0}$ and $g\left( E_{4},\phi T\right) =\sin \alpha
_{0}.$

So we can state:

\begin{theorem}
Let $M(c)=(M^{2n+1},\phi ,\xi ,\eta ,g)$ be a Sasakian space form with $%
c\neq 1,$ $n\geq 2$ and $\gamma :I\subset
\mathbb{R}
\longrightarrow M(c)$ a Legendre curve of osculating order $r$ such that $%
g(\phi T,E_{2})\neq 0,1,-1$.

$(1)$ If $c\leq -3$ and $\frac{\delta _{1}}{\delta _{2}}$ $\geq 0$, then $%
\gamma $ is interpolating sesqui-harmonic if and only if it is a geodesic.

$(2)$ If $c>-3$ and $\frac{\delta _{1}}{\delta _{2}}$ $<0$, then $\gamma $
is interpolating sesqui-harmonic if and only if $\phi T=\cos \alpha
_{0}E_{2}+\sin \alpha _{0}E_{4},$
\begin{equation*}
k_{1},k_{2},k_{3}=\text{constant}>0,
\end{equation*}%
\begin{equation*}
k_{1}^{2}+k_{2}^{2}=\frac{c+3}{4}+\frac{3\left( c-1\right) }{4}\cos
^{2}\alpha _{0}-\frac{\delta _{1}}{\delta _{2}},
\end{equation*}%
\begin{equation*}
k_{2}k_{3}=-\frac{3\left( c-1\right) }{8}\sin 2\alpha _{0},
\end{equation*}

where $\alpha _{0}\in \left( 0,2\pi \right) \backslash \left\{ \frac{\pi }{2}%
,\pi ,\frac{3\pi }{2}\right\} $ is constant such that $\left( c+3+3\left(
c-1\right) \cos ^{2}\alpha _{0}\right) \delta _{2}-4\delta _{1}>0$ and $%
3\left( c-1\right) \sin 2\alpha _{0}<0.$
\end{theorem}

\begin{Remark}
For $c\neq 1$ and $g(\phi T,E_{2})\neq 0,1,-1$, there are also interpolating
sesqui-harmonic curves which are not helices.
\end{Remark}

Now, we give brief information about the Sasakian space form $%
\mathbb{R}
^{2n+1}(-3)$ \cite{Blair2}:

Let us take $M=%
\mathbb{R}
^{2n+1}$ with the standard coordinate functions $\left(
x_{1},...,x_{n},y_{1},...,y_{n},z\right) ,$ the contact structure $\eta =%
\frac{1}{2}(dz-\sum\limits_{i=1}^{n}y_{i}dx_{i}),$ the characteristic vector
field $\xi =2\frac{\partial }{\partial z}$ and the tensor field $\phi $
given by%
\begin{equation*}
\phi =%
\begin{bmatrix}
0 & \delta _{ij} & 0 \\
-\delta _{ij} & 0 & 0 \\
0 & y_{j} & 0%
\end{bmatrix}%
.
\end{equation*}%
The Riemannian metric is $g=\eta \otimes \eta +\frac{1}{4}%
\sum\limits_{i=1}^{n}\left( (dx_{i})^{2}+(dy_{i})^{2}\right) .$ Thus, $%
\mathbb{R}
^{2n+1}(-3)$ is a Sasakian space form with constant $\phi -$sectional
curvature $c=-3$. The vector fields%
\begin{equation}
X_{i}=2\frac{\partial }{\partial y_{i}},\text{ }X_{i+n}=\phi X_{i}=2(\frac{%
\partial }{\partial x_{i}}+y_{i}\frac{\partial }{\partial z}),\text{ }1\leq
i\leq n,\text{ }\xi =2\frac{\partial }{\partial z},  \label{4}
\end{equation}%
form a $g$-orthonormal basis and Levi-Civita connection is obtained as%
\begin{equation}
\nabla _{X_{i}}X_{j}=\nabla _{X_{i+n}}X_{j+n}=0,\text{ }\nabla
_{X_{i}}X_{j+n}=\delta _{ij}\xi ,\text{ }\nabla _{X_{i+n}}X_{j}=-\delta
_{ij}\xi ,  \label{LC1}
\end{equation}%
\begin{equation}
\nabla _{X_{i}}\xi =\nabla _{\xi }X_{i}=-X_{n+i},\text{ }\nabla
_{X_{i+m}}\xi =\nabla _{\xi }X_{i+n}=X_{i}  \label{LC2}
\end{equation}%
(see \cite{Blair}).

Now, we give an example for interpolating sesqui-harmonic Legendre curves in
$%
\mathbb{R}
^{5}(-3):$

\bigskip \textbf{Example. }Let $\gamma =(\gamma _{1},...,\gamma _{5})$ be a
unit speed Legendre curve in $%
\mathbb{R}
^{5}(-3).$ We can write the tangent vector field $T$ of
\begin{equation*}
\gamma T=\frac{1}{2}\left\{ \gamma _{3}^{\prime }X_{1}+\gamma _{4}^{\prime
}X_{2}+\gamma _{1}^{\prime }X_{3}+\gamma _{2}^{\prime }X_{4}+\left( \gamma
_{5}^{\prime }-\gamma _{1}^{\prime }\gamma _{3}-\gamma _{2}^{\prime }\gamma
_{4}\right) \xi \right\} .
\end{equation*}%
Using the above equation, $\eta (T)=0$ and $g(T,T)=1,$ we have%
\begin{equation*}
\gamma _{5}^{\prime }=\gamma _{1}^{\prime }\gamma _{3}+\gamma _{2}^{\prime
}\gamma _{4}
\end{equation*}%
and%
\begin{equation*}
(\gamma _{1}^{\prime })^{2}+...+(\gamma _{5}^{\prime })^{2}=4.
\end{equation*}%
So for a Legendre curve (\ref{LC1}), (\ref{LC2}) and (\ref{4}) gives us%
\begin{equation}
\nabla _{T}T=\frac{1}{2}\left( \gamma _{3}^{\prime \prime }X_{1}+\gamma
_{4}^{\prime \prime }X_{2}+\gamma _{1}^{\prime \prime }X_{3}+\gamma
_{2}^{\prime \prime }X_{4}\right) ,  \label{5}
\end{equation}%
and%
\begin{equation}
\phi T=\frac{1}{2}\left( -\gamma _{1}^{\prime }X_{1}-\gamma _{2}^{\prime
}X_{2}+\gamma _{3}^{\prime }X_{3}+\gamma _{4}^{\prime }X_{4}\right) .
\label{6}
\end{equation}%
From (\ref{5}) and (\ref{6}), $\phi T\perp E_{2}$ if and only if%
\begin{equation*}
\gamma _{1}^{\prime }\gamma _{3}^{\prime \prime }+\gamma _{2}^{\prime
}\gamma _{4}^{\prime \prime }=\gamma _{3}^{\prime }\gamma _{1}^{\prime
\prime }+\gamma _{4}^{\prime }\gamma _{2}^{\prime \prime }.
\end{equation*}%
So we can state the following example:

Let us take $\gamma (t)=\left( \sin 2t,-\cos 2t,0,0,1\right) $ in $%
\mathbb{R}
^{5}(-3).$ By the use of Theorem \ref{Theo 3.1} and the above equations, $%
\gamma $ is an interpolating sesqui-harmonic Legendre curve with osculating
order $r=2$, $k_{1}=2$, $\delta _{1}=-8$, $\delta _{2}=2$ and $\phi T\perp
E_{2}$. We can see that Theorem \ref{Theo 3.1} are verified. From the
equations (3-1) in \cite{FO}, the curve $\gamma $ is not biharmonic. Hence
the biharmonicity and interpolating sesqui-harmonic of $\gamma $ are
different.

Fatma KARACA

Beykent University,

Department of Mathematics,

34550, Beykent, Buyukcekmece,

Istanbul, TURKEY.

E-mail: fatmagurlerr@gmail.com

\medskip

Cihan \"{O}ZG\"{U}R

Bal\i kesir University,

Department of Mathematics,

10145, \c{C}a\u{g}\i s, Bal\i kesir, Turkey

E-mail: cozgur@balikesir.edu.tr

\medskip

Uday Chand DE

Department of Pure Mathematics,

University of Calcutta 35,

Ballygunge Circular Road,

Kolkata 700019, West Bengal, India

E-mail: uc\_de@yahoo.com

\end{document}